\documentclass[reqno]{amsart}
\usepackage{amsmath, amssymb}
\usepackage{amsthm, epsfig, latexsym}
\usepackage{hyperref}
\usepackage[mathscr]{euscript}

\usepackage{color}
\usepackage{setspace}
\usepackage{refcheck} 
\norefnames
\nocitenames

 \newtheorem{theorem}{Theorem}
 \newtheorem{lemma}[theorem]{Lemma}
 \newtheorem{proposition}[theorem]{Proposition}
 \newtheorem{corollary}[theorem]{Corollary}
 \theoremstyle{definition}

 \theoremstyle{remark}

 \newcommand{\mc}{\mathcal}
 
\newcommand{\Tau}{\mathcal{T}}

 \newcommand{\B}{\mc{B}}
 
 \newcommand{\D}{\mc{D}}

 \newcommand{\G}{\mc{G}}

 \newcommand{\Sg}{\mc{S}}

 \newcommand{\M}{\mc{M}}

 \newcommand{\C}{\mathbb{C}}
 \newcommand{\R}{\mathbb{R}}

 \newcommand{\Z}{\mathbb{Z}}

 \newcommand{\p}{\varphi}

\newcommand{\dz}{\text{\rm d}z}
 \newcommand{\dt}{\text{\rm d}t}

 \newcommand{\dx}{\text{\rm d}x}
 
 \newcommand{\dy}{\text{\rm d}y}

    \renewcommand{\d}{\text{\rm d}}

\newcommand{\es}[1]{\begin{equation}\begin{split}#1\end{split}\end{equation}}
\newcommand{\est}[1]{\begin{equation*}\begin{split}#1\end{split}\end{equation*}}

\newcommand{\ov}{\overline}
\renewcommand{\H}{\mc{H}}
\newcommand{\im}{{\rm Im}\,}
\newcommand{\re}{{\rm Re}\,}

\newcommand{\ep}{\varepsilon}

\newcommand{\al}{\alpha}

 \def\today{\ifcase\month\or
  January\or February\or March\or April\or May\or June\or
  July\or August\or September\or October\or November\or December\fi
  \space\number\day, \number\year}

\title{Interpolation Formulas with derivatives in de Branges Spaces II}
\author{Felipe Gon\c{c}alves and Friedrich Littmann}
\address{University of Alberta, Mathematical and Statistical Sciences, 632 Central Academic Building, Edmonton, Alberta, Canada T6G 2G1.}
\email{felipe.goncalves@ualberta.ca}

\address{Department of mathematics, North Dakota State University, Fargo, ND 58105-5075.}
\email{friedrich.littmann@ndsu.edu}
\begin{document}
 
\begin{abstract} We investigate necessary and sufficient conditions under which entire functions in de Branges spaces can be recovered from function values and values of derivatives. Our main focus is on spaces with a structure function whose logarithmic derivative is bounded in the upper half-plane.
\end{abstract}

\maketitle

\numberwithin{equation}{section}


\section{Introduction}

One of the classical problems in complex analysis is to reconstruct an entire function from a countable set of data.  This article considers reconstruction of an entire function  $F$ from its values and the values of its derivatives up to a specified order at a discrete set of points on the real line. The entire function $F$ will be an element of a reproducing kernel Hilbert space  in the sense of de Branges \cite{dB} (we review the pertinent facts of de Branges spaces in Section \ref{dB-spaces}). We are interested in reconstruction formulas that converge to $F$ in the norm of the space. 
\smallskip

To motivate the results, we briefly describe the situation in the classical Paley-Wiener space. Given  $\tau>0$ the Paley-Wiener space $PW(\tau)$ is defined as the space of entire functions of exponential type at most $\tau$ such that their restriction to the real axis belongs to $L^2(\R)$. These are special spaces with a reproducing kernel structure. The reproducing kernel of $PW(\tau)$ is given by
$$
K(w,z)= \frac{\sin \tau (z-\ov w)}{\pi(z-\ov w)}.
$$
and
$$
F(w) = \int_\R F(x)\ov{K(w,x)}\dx,
$$
for every $F\in PW(\tau)$.
\smallskip

A basic result of the theory of Paley-Wiener spaces is that for all $F\in PW(\tau)$ we have
\es{\label{PW-rec-form}
F(z) = \frac{\sin (\tau z)}{\pi}\sum_{n\in\Z} \frac{(-1)^nF(\pi n/\tau)}{z-\pi n/\tau},
}
where the sum converges in $L^2(\R)$ as well as uniformly on compact sets of $\C$ (this is sometimes called the Shannon-Whittaker interpolation formula). The existence of interpolation formulas using derivatives is also known in Paley-Wiener spaces. In \cite[Theorem 9]{V}, J. Vaaler proved that
\es{\label{PW-rec-form-der}
F(z) = \bigg(\frac{\sin \tau x}{\tau}\bigg)^2\sum_{n\in\Z} \left(\frac{F(\pi n/\tau)}{(z-\pi n/\tau)^2} +\frac{F'(\pi n/\tau)}{z-\pi n/\tau}\right),
}
for every $F\in PW(2\tau)$.

\smallskip

The Paley-Wiener space is a particular case of a wider class of spaces of entire functions called de Branges spaces  (see \cite[Chapter 2]{dB} for the case $p=2$ and \cite{Bar,Bar2} for $p\neq 2$). A generalization of Vaaler's result in de Branges spaces was obtained by the first named author in \cite{Gon} which we describe next. For a given Hermite-Biehler function $E$, a de Branges space $\H(E^2)$ can be constructed where the intended interpolation formulas will take place. With $E^*(z) = \overline{E(\bar{z})}$ we define $A(z) = 2^{-1}(E(z) +E^*(z))$ and $B(z) = (i/2) (E(z) - E^*(z))$. Furthermore, $\varphi$ denotes the phase of $E$, that is, $E(x)e^{i\varphi(x)} \in \R$ for all real $x$. We denote by $\mc{T}_B$ the set of (real) zeros of $B$. 

 Throughout this paper we use for $f,g$ in a set $\mc{F}$ the notation $f\simeq g$ to mean that there exist positive constants $C,D$ with $C f\le g \le D f$, and the constants are uniform in $\mc{F}$. Similarly, $f\lesssim g$ stands for the statement $f \le C g$ with uniform $C>0$. 

\medskip

\noindent{\bf Theorem A} (\cite[Theorem 1]{Gon}). {\it Let $E$ be a Hermite-Biehler function such that $\mc{H}(E^2)$ is a de Branges space closed under differentiation, $AB\notin \mc{H}(E^2)$ and $\varphi'(t)\ge \delta>0$ for some $\delta>0$ and all $t\in \mc{T}_B$. Then
\[
F(z) = B(z)^2 \sum_{t\in \mc{T}_B} \left(\frac{F(t)}{B'(t)^2(z-t)^2} + \frac{F'(t) B'(t) - F(t) B''(t)}{B'(t)^3(z-t)} \right),
\]
and the series converges in $\mc{H}(E^2)$ as well as uniformly on compact subsets of $\C$. 
}
\medskip

This result immediately leads to questions regarding necessity and sufficiency of the assumptions under which the conclusions of the theorem hold. Our main results provide answers to most of these questions.

\begin{enumerate}

\item The requirement that $\mc{H}(E^2)$ is closed under the differentiation operator is a natural condition, since we use derivatives for reconstruction, but it is usually hard to check directly. A sufficient criterion was given by A.\ Baranov \cite[Theorem 3.2]{Bar} who showed that if $E'/E\in H^\infty(\C^+)$ (see Section \ref{Lp-case}), then differentiation defines a bounded operator on $\mc{H}(E)$. He also showed that the reverse implication is not true (see \cite[Section 4.2]{Bar}). Our first result connects here, we prove in Theorem \ref{Thm1} that $E'/E\in H^\infty(\C^+)$ is actually equivalent to $\mc{H}(E^2)$ being closed under differentiation.

\item The uniform lower bound for  $\varphi'$ is necessary. We show in Theorem \ref{Thm2} that  a relation of the form
\begin{align}\label{frame-intro}
  \sum_{t\in \mc{T}_B } \frac{|F(t)|^2+|F'(t)|^2}{K(t,t)} \simeq \|F\|_{\mc{H}(E^2)}^2
\end{align}
holds for $F\in \mc{H}(E^2)$ (with constants independent of $F$), and we prove in Theorem \ref{Thm3} that if the series in \eqref{frame-intro} bounds the norm, then it follows already that $\varphi'\ge \delta$ on $\mc{T}_B$ for some positive $\delta$.  This is essentially due to the fact that the norm of $B(z)^\nu/(z-t)$  in $\mc{H}(E^\nu)$ and the value of its $(\nu-1)$st derivative at $z=t$ can be estimated by distinct powers of $\varphi'(t)$ if $\nu\ge 2$. We remark further that boundedness of the differentiation operator on $\mc{H}(E^\nu)$ for $\nu\ge 2$ implies that $\varphi'$ is bounded from above on $\R$. We provide a proof in Proposition \ref{upper-phase-bound} (it is worth emphasizing that none of these implications are valid for $\nu=1$).

\item Reconstruction formulas for all derivatives hold. Theorem A suggests that reconstruction from $\{F^{(j)}(t): t\in \mc{T}_B, 0\le j\le \nu-1\}$ requires us to work in $\mc{H}(E^\nu)$ which is the setting for this article (a fact already suggested by Vaaler's reconstruction \eqref{PW-rec-form-der}, where in order to maintain the same interpolation nodes the type was doubled). We give the precise formulation in Theorem \ref{Thm2}. 

\item The condition $AB\notin \mc{H}(E)$ does not need to be assumed, it will follow if differentiation is a bounded operator on $\mc{H}(E^\nu)$ for $\nu\ge 2$. This is shown in Corollary \ref{ABcorollary}.

\item An interesting and mainly still open question concerns the existence of corresponding formulas in $\mc{H}^p(E^\nu)$ where $1<p<\infty$ (the $L^p$ version of de Branges spaces). We include a few results in this direction in Section \ref{Lp-case}.

\end{enumerate}

For reconstruction with derivatives one needs to work in $\mc{H}(E^\nu)$ rather than in $\mc{H}(E)$. This is something of a drawback, in practice  it would be preferable to have interpolation formulas that converge in the norm of $L^2(\R,\mu)$ ($\mu$ some non-negative measure). On the other hand, the structure of the interpolation series leads naturally to convergence in $\mc{H}(E^\nu)$. This leads to the question when there exists $E$ such that the restriction of the norm of $L^2(\R,\mu)$ to functions of exponential type $\tau$ and the norm on $\mc{H}(E^\nu)$ are equivalent. This is largely unknown, except in special cases. For example, \cite{Gon} includes calculations for homogeneous spaces and we extend those results in Corollary \ref{hom-rec}. 

We mention that formulas like \eqref{PW-rec-form} and \eqref{PW-rec-form-der} have been used to construct extremal functions related to Beurling-Selberg extremal problems. These are functions $F$ of prescribed exponential type that minimize the $L^1(\R,\mu)$-distance ($\mu$ is some non-negative measure) from a given function $g$ and such that $F$ lies below or above $g$ on $\R$. These functions have the special property that they interpolate the target function $g$ (an its derivative) at a certain sequence of real points and have several special properties that are very useful in applications to analytic number theory, being the key to provide sharp (or improved) estimates. For instance, in connection to: large sieve inequalities \cite{HV, V}, Erd\"{o}s-Tur\'{a}n inequalities \cite{CV2,V}, Hilbert-type inequalities \cite{CL3, CLV, CV2, GV, V}, Tauberian theorems \cite{GV} and bounds in the theory of the Riemann zeta-function and general $L$-functions \cite{CCLM,CCM,CCM2,CF,CS,Ga,GG}. Further constructions and applications can also be found in \cite{CC,CG,CL4,GMK,LSp}.


\section{Results}

\subsection{De Branges spaces}\label{dB-spaces}
Since our results are formulated in the language of de Branges spaces, we briefly review the basics (see \cite[Chapter 2]{dB}).  A function $F$ analytic in the open upper half-plane 
$$\C^+ = \{z \in \C:\ \im z >0\}$$ 
has { bounded type} if it can be written as the quotient of two functions that are analytic and bounded in $\C^+$. If $F$ has bounded type in $\C^+$ then, according to \cite[Theorems 9 and 10]{dB}, we have
\es{\label{mean-type}
\limsup_{y \to \infty} \, y^{-1}\log|F(iy)| = v(F) <\infty.
}
The number $v(F)$ is called the { mean type} of $F$. We say that an entire function $F:\C \to \C$, not identically zero, has { exponential type} if 
\begin{equation*}
\limsup_{|z|\to \infty} |z|^{-1}\log|F(z)| = \tau(F) <\infty.
\end{equation*}
In this case, the non-negative number $\tau(F)$ is called the { exponential type of $F$}. If $F:\C \to \C$ is entire we define $F^*:\C \to \C$ by $F^*(z) = \overline{F(\overline{z})}$ and if $F(z)=F^*(z)$ we say that it is { real entire}.

A { Hermite-Biehler function} $E:\C \to \C$ is an entire function that satisfies the fundamental inequality
\begin{align*}
|E^*(z)|<|E(z)|,
\end{align*}
for all $z\in\C^+$. We define the {de Branges space} $\H(E)$ to be the space of entire functions $F:\C \to \C$ such that
\begin{equation*}
\|F\|_E^2 := \int_{-\infty}^\infty |F(x)|^{2} \, |E(x)|^{-2} \, \dx <\infty\,,
\end{equation*}
and such that $F/E$ and $F^*/E$ have bounded type and non-positive mean type in $\C^+$. This forms a Hilbert space with respect to the inner product 
\begin{equation*}
\langle F, G \rangle_E := \int_{-\infty}^\infty F(x) \,\overline{G(x)} \, |E(x)|^{-2} \, \dx.
\end{equation*}
The Hilbert space $\H(E)$ has the special property that, for each $w \in \C$, the map $F \mapsto F(w)$ is a continuous linear functional on $\H(E)$. Therefore, there exists a function $z \mapsto K(w,z)$ in $\H(E)$ such that 
\begin{equation}\label{Intro_rep_pro}
F(w) = \langle F, K(w,\cdot) \rangle_E\,.
\end{equation} 
The function $K(w,z)$ is called the { reproducing kernel} of $\H(E)$. If we write
\begin{equation*}
A(z) := \frac12 \big\{E(z) + E^*(z)\big\} \ \ \ {\rm and}  \ \ \ B(z) := \frac{i}{2}\big\{E(z) - E^*(z)\big\},
\end{equation*}
then $A$ and $B$ are real entire functions with only real zeros and $E(z) = A(z) -iB(z)$. The reproducing kernel is then given by \cite[Theorem 19]{dB}
\begin{equation*}
\pi (z - \ov{w})K(w,z) = B(z)A(\ov{w}) - A(z)B(\ov{w}),
\end{equation*}
or alternatively by
\begin{equation}\label{Intro_Def_K_E}
2\pi i (\ov{w}-z)K(w,z) = E(z)E^*(\ov{w}) - E^*(z)E(\ov{w}). 
\end{equation}
When $w = \ov{z}$ we have
\begin{equation}\label{Intro_Def_K_AB}
\pi K(\ov{z}, z) = B'(z)A(z) - A'(z)B(z).
\end{equation}
We may apply the Cauchy-Schwarz inequality in \eqref{Intro_rep_pro} to obtain that
\begin{align}\label{CS-ineq}
|F(w)|^2 \leq \|F\|^2_EK(w,w),
\end{align}
for every $F\in\H(E)$. Also, it is not hard to show  that $K(w,w)=0$ if and only if $w \in\R$ and $E(w) = 0$ (see \cite[Lemma 11]{HV}).

We denote by $\varphi$ a phase function associated to $E$. This is an analytic function in  a neighborhood of $\R$ defined by the condition $e^{i\varphi(x)} E(x) \in\R$ for all real $x$ ($\p$ is uniquely defined modulo the sum of a constant of the form $\pi k$ for $k\in\Z$). We obtain that
\es{\label{phi-id}
\varphi'(x)=\re \bigg\{i\frac{E'(x)}{E(x)}\bigg\}=\pi\frac{K(x,x)}{|E(x)|^2} > 0,
}
for all real $x$ and thus $\p(x)$ is an increasing function of real $x$. We also have that
$$
e^{2i\p(x)} = \frac{A(x)^2}{|E(x)|^2}  - \frac{B(x)^2}{|E(x)|^2} + 2i\frac{A(x)B(x)}{|E(x)|^2},
$$
for all real $x$. As a consequence, the points $t\in\R$ such that $\p(t)\equiv 0 \mod \pi$ coincide with the real zeros of $B/E$ and the points $s\in\R$ such that $\p(s)\equiv \pi/2\mod \pi$ coincide with the real zeros of $A/E$ and by \eqref{phi-id}, these zeros are simple. In particular, the function $B/A$ has only simple real zeros and simple real poles that interlace.

Denote by $\mc{T}_B$ the set of real zeros of the function $B$. This set plays a special role in the theory of de Branges associated with a function $E$ with no real zeros. In this case, by \eqref{Intro_Def_K_AB} and the reproducing kernel property, we easily see that the functions $\{B(z)/(z-t)\}_{t\in\Tau_B}$ form an orthogonal set in $\H(E)$ and, by \cite[Theorem 22]{dB} form a basis if and only if $B\notin \H(E)$. In that case the identities
\es{\label{norm-HE}
\|F\|^2_{\H(E)} = \sum_{t\in\mc{T}_B} \frac{|F(t)|^2}{K(t,t)}
}
and 
$$
F(z) = B(z)\sum_{t\in\Tau_B}\frac{F(t)}{B'(t)(z-t)}
$$
hold for all $F\in \H(E)$.  

Finally, we say that $\H(E)$ is closed under differentiation if $F' \in \H(E)$ whenever $F\in\H(E)$. Inequality \eqref{CS-ineq} together with the fact that $w\in\C\mapsto K(w,w)$ is a continuous functions implies that convergence in the norm of $\H(E)$ implies uniform convergence in compact sets of $\C$. As a consequence, differentiation defines a closed linear operator on $\H(E)$ and therefore by the Closed Graph Theorem defines a bounded linear operator on $\H(E)$. 

\subsection{Main Results} 

We prove in Section  \ref{Thm1-proof-section}  the following statement about the differentiation operator on $\mc{H}(E^\nu)$ for $\nu \ge 2$. 

\begin{theorem}\label{Thm1} The following statements are equivalent:
\begin{enumerate}
\item $E'/E\in H^\infty(\C^+)$;
\item $\mc{H}(E^\nu)$ is closed under differentiation for some integer $\nu \ge 2$;
\item $\mc{H}(E^\nu)$ is closed under differentiation for all integers $\nu \ge 1$. 
\end{enumerate}
\end{theorem}

The next definitions set up the interpolation formula in $\mc{H}(E^\nu)$. Let $\nu\geq 2$ be an integer and $E$ be a Hermite-Biehler function with no real zeros (hence the zeros of $B$ are simple). Denote by $A_\nu$ and $B_\nu$ the real entire functions that satisfy $E^\nu = A_\nu - iB_\nu$ and by $K_\nu(w,z)$ the reproducing kernel associated with $\H(E^\nu)$. We define the collection $\mc{B}_\nu$ of functions $z\mapsto B_{\nu,j}(z,t)$ given by
\begin{align}\label{intro-Bdef}
B_{\nu,j}(z,t) = \frac{B(z)^\nu}{(z-t)^j}
\end{align}
where $t\in \mc{T}_B$ and $1\le j\le \nu$. For $\ell\geq 0$ we denote by $P_{\nu,\ell}(z,t)$ the Taylor polynomial of degree $\ell$ of $B_{\nu,\nu}(z,t)^{-1}$  expanded into a power series at $z=t$ as a function of $z$, that is,
\es{\label{1/B-P-def}
\frac{1}{B_{\nu,\nu}(z,t)} = \underbrace{\sum_{n =0}^\ell a_{\nu,n}(t)(z-t)^n}_{P_{\nu,\ell}(z,t)} + \sum_{n >\ell} a_{\nu,n}(t)(z-t)^n .
}
Finally, we denote by $\mc{G}_\nu$ the collection of functions $z\mapsto G_{\nu,j}(z,t)$ defined by
\begin{align}\label{intro-Gdef}
G_{\nu,j}(z,t) = B_{\nu,\nu-j}(z,t)  \frac{P_{\nu,\nu-j-1}(z,t)}{j!},
\end{align}
for $j=0,...,\nu-1$ and $t\in\Tau_B$. We note that
$$
G_{\nu,j}(z,t)=\frac{(z-t)^j}{j!} -\frac{B(z)^\nu}{j!} \sum_{n\geq \nu-j} a_{\nu,n}(t)(z-t)^{n+j-\nu}.
$$
We easily see that these functions satisfy the following property
\es{\label{G-delta-property}
G^{(\ell)}_{\nu,j}(s,t) = \delta_0(s-t)\delta_0(\ell-j)
}
for $\ell,j=0,...,\nu-1$ and $s,t\in\Tau_B$ ($G^{(\ell)}_{\nu,j}(s,t)=\partial_z^\ell G_{\nu,j}(z,t)|_{z=s}$). We prove in Section \ref{derivative-frames} the following statement.

\begin{theorem}\label{Thm2} 
Let $E$ be a Hermite-Biehler function such that $E'/E\in H^\infty(\C^+)$. Assume that   the phase function $\p$ associated with $E$ satisfies $1\lesssim \p'(t)$ for $t\in\Tau_B$. Then for $\nu\geq 2$ the following statements hold:
\begin{enumerate}
\item For every $F\in \H(E^\nu)$
\begin{align}\label{intro-Gdual}
F(z) = \sum_{t\in \mc{T}_B} \sum_{j=0}^{\nu-1} F^{(j)}(t) G_{\nu,j}(z,t),
\end{align}
where the series converges to $F$ in the norm of $\H(E^\nu)$ $($hence, also uniformly in compact subsets of $\C)$;

\item For $F\in \mc{H}(E^\nu)$ 
\begin{align}
 \sum_{t\in\mc{T}_B}  \sum_{j=0}^{\nu-1} \frac{ |F^{(j)}(t)|^2}{K_\nu(t,t)}  \simeq \|F\|_{\mc{H}(E^\nu)}^2\label{intro-Dframe}
\end{align}
and
\begin{align}
  \sum_{t\in\mc{T}_B}  \sum_{j=0}^{\nu-1}  K_\nu(t,t) |\langle F, G_{\nu,j}(\cdot,t)\rangle_{E^\nu} |^2   \simeq \|F\|_{\mc{H}(E^\nu)}^2 \label{intro-Gframe},
\end{align}
where the implied constants are independent of $F$;

\item If any of the terms of the series in \eqref{intro-Dframe} $($respect. \eqref{intro-Gframe}$)$ is removed, the modified formula fails to hold .
\end{enumerate}
\end{theorem}

The necessity of the boundedness of the phase at points in $\mc{T}_B$ is the focus of the following statement.

\begin{theorem}\label{Thm3} Let $E$ be a Hermite-Biehler function such that $E'/E\in H^\infty(\C^+)$, and let $\nu\ge 2$. If
\[
\|F\|^2_{\mc{H}(E^\nu)}\lesssim  \sum_{t\in\mc{T}_B}  \sum_{j=0}^{\nu-1} \frac{ |F^{(j)}(t)|^2}{K_\nu(t,t)},
\]
for all $F\in \mc{H}(E^\nu)$, then $1\lesssim \p'(t)$ for $t\in \mc{T}_B$. 
\end{theorem}

Finally, we remark that there exist Hermite-Biehler functions $E$ with $E'/E\in H^\infty(\C^+)$ and $\p'(x)\to 0$ as $|x|\to \infty$. For example, let $\ov w_n = x_n - iy_n$ with $y_n>0$ be such that
\[
\sum_{n} \left(\frac{1}{y_n} + \frac{|x_n|}{x_n^2 + y_n^2}\right)<\infty
\]
and define
\[
E(z) = \prod_{n} \left(1-\frac{z}{\ov w_n}\right) e^{z \re (1/w_n)}.
\]
The identity $e^{2i\varphi(x)} = E^*(x)/E(x)$, valid for $x\in\R$, convergence of $\sum_n y_n^{-1}$, and the representation
\[
\p'(x) = \sum_n \frac{y_n}{(x-x_n)^2+y_n^2}
 \]
  may be used to show that $\p'(x)\to 0$ as $|x|\to \infty$. Logarithmic differentiation of $E(z)$ for $\im z >0$ gives 
  \[
  \bigg|\frac{E'(z)}{E(z)}\bigg| = \bigg|\sum_{n} \frac{1}{z-\ov w_n} + \frac{x_n}{x_n^2+y_n^2}\bigg| \leq  \sum_n \frac{1}{y_n} + \frac{|x_n|}{x_n^2+y_n^2} < \infty,
  \]
 that is, $E'/E$ is uniformly bounded in  $\C^+$. This example shows that the condition $E'/E\in H^\infty(\C^+)$ alone is not strong enough to imply the statements of Theorem \ref{Thm2}.

\subsection{Applications}
We mention finally the following two applications. As was already pointed out in \cite{Gon}, Theorem \ref{Thm2} gives a sufficient condition for complete interpolating sequences with derivatives in Paley-Wiener space. We say that a sequence of real numbers $\{\lambda_m\}$ is {complete interpolating with derivatives up to order $\nu-1$} in the Paley-Wiener space $PW=PW(\tau)$, if the system
$$
F^{(n)}(\lambda_m) = f_{n,m}
$$
has a {unique} solution $F\in PW$ for all 
$$
(\{f_{0,m}\},\{f_{1,m}\},...,\{f_{\nu-1,m}\})\in \overbrace{\ell^2(\Z)\times...\times\ell^2(\Z)}^{\nu\,\,\text{times}}.
$$
We note that this condition is equivalent to
$$
\|F\|^2_{L^2} \simeq \sum_{n=0}^{\nu-1} \sum_{m} |F^{(n)}(\lambda_m)|^2,
$$
for all $F\in PW$, and the above norm equivalence does not hold if one node $\lambda_{m_0}$ is removed form the sequence $\{\lambda_m\}$.

\begin{corollary}
Assume that  $PW=\H(E^\nu)$ $($as sets$)$ and $\inf_{t\in\Tau_B}{\p'(t)}>0$. Then $\Tau_B$ is a complete interpolating sequence with derivatives up to order $\nu-1$ for $PW$.
\end{corollary}

We now describe another application of our main results. There is a variety of examples of de Branges spaces that can be found in \cite[Chapter 3]{dB}. A basic example would be the classical Paley-Wiener space $\H(e^{-i\tau z})$  which consists of entire functions of exponential type at most $\tau$ that have finite $L^2(\R)$-norm when restricted to the real axis. A natural extension of these spaces by using power weights in the $L^2$-norm is the so called homogeneous de Branges spaces. In what follows we briefly review the main aspects these spaces (see also \cite[Section 5]{HV}).

For every real number $\al>-1$ let $J_\al$ denote the classical Bessel function of the first kind. Let 
\est{
E_\alpha(z) = A_\alpha(z)-iB_\alpha(z),
}
where
\est{
A_\alpha(z) & = \Gamma(\alpha+1)(z/2)^{-\alpha}J_\alpha(z), \\
B_\alpha(z) & = \Gamma(\alpha+1)(z/2)^{-\alpha}J_{\alpha+1}(z).
}
It is known that $E_\al$ is a Hermite-Biehler function of bounded type in $\C^+$ and of exponential type exactly $1$. Due to well known estimates for Bessel functions we can deduce that
\es{\label{E-est}
|E_\alpha(x)|\,|x|^{\alpha+1/2} \simeq 1 , \ \ \text{for } 1\lesssim |x|.
}
The defining differential equation of $J_\al$ can also be used to deduce the following equations
\begin{align*}
\begin{split}
A'_\alpha(z)&=-B_\alpha(z), \\ B'_\alpha(z)&=A_\alpha(z)-(2\alpha+1)B_\alpha(z)/z.
\end{split}
\end{align*}
In particular we have
\begin{equation*}
i\frac{E_\alpha'(z)}{E_\alpha(z)}=1 - (2\alpha+1)\frac{B_\alpha(z)}{zE_\alpha(z)},
\end{equation*}
and we conclude that $E_\al'/E_\al\in H^\infty(\C^+)$. Also, since
$$
\p'_\al(x) = \re \bigg[i\frac{E_\alpha'(x)}{E_\alpha(x)}\bigg] = 1 - (2\alpha+1)\frac{A_\al(x)B_\alpha(x)}{x|E_\alpha(x)|^2}
$$
for real $x$, we deduce that $\lim_{|x|\to \infty} \p'_\al(x)=1$ ($\p_\al$ being the phase associated with $E_\al$). This in turn implies that
$$
K_{\al,\nu}(x,x)^{-1} \simeq |E(x)|^{-2\nu} \simeq |x|^{(2\al+1)\nu}
$$
for real $x$ such that $1\lesssim |x|$, where $K_{\al,\nu}(w,z)$ is the reproducing kernel of $\H(E_\al^\nu)$ and $\nu\geq 1$ is an integer.

The space $\H(E_\al^\nu)$ is then well-defined and by \cite[Lemma 12]{HV} and \eqref{E-est} it coincides with the space of entire functions $F$ of exponential type at most $\nu$ and such that
\es{\label{norm-hom}
\int_{|x|\geq 1}|F(x)|^2 \omega_\al(x)^{\nu}\dx <\infty,
}
where $\omega_\al(x)=|x|^{2\al+1}$ if $|x|\geq 1$ and $\omega_\al(x)=1$ if $|x|\leq 1$. Let $\{a_{\al,n}\}_{n\in\Z\setminus \{0\}}$ denote the real zeros of $A_\al$ (note that these are exactly the non-zero Bessel zeros of $J_\al$).

\begin{corollary}\label{hom-rec}
Let $\alpha>-1$ be a real number and $\nu\geq 1$ be an integer. Let $F$ be an entire function of exponential type at most $\nu$ such that \eqref{norm-hom} is finite. Then we have
\est{
\int_{\R}|F(x)|^2\omega_{\alpha}(x)^{\nu}\dx \simeq  \sum_{n\in\Z\setminus \{0\}} \sum_{j=0}^{\nu-1}|F^{(j)}(a_{\alpha,n})|^2|a_{\al,n}|^{(2\al+1)\nu},
}
where the implied constant does not depend on $F$. Moreover, there exists functions $\{G_{\al,\nu,n}\}_{n\in\Z\setminus \{0\}}$ such that
$$
F(z) = \sum_{n\in\Z\setminus \{0\}} \sum_{j=0}^{\nu-1} F^{(j)}(a_{\alpha,n}) G_{\al,\nu,n}(z),
$$
where the series converges to $F(z)$ in the norm \eqref{norm-hom} and uniformly in compact sets of $\C$.
\end{corollary}

\section{Acknowledgments}
F. Gon\c{c}alves acknowledges the support from the Science Without Borders program CNPq-Brazil grant 201697/2014-9 and the University of Alberta StartUp Funds. Part of this work was completed during two visits to North Dakota State University, for which he gratefully acknowledges the support.


\section{Differentiation operator on $\mc{H}(E^\nu)$}\label{Thm1-proof-section}

We prove Theorem \ref{Thm1} in this section. We require some auxiliary results about spaces $\mc{H}(E^\nu)$ on which the derivative defines a bounded operator. Throughout this section $E$ denotes a Hermite-Biehler entire function with phase $\varphi$. 

\begin{proposition}\label{upper-phase-bound} Let $\nu\ge 2$ be an integer. Assume that differentiation defines a bounded operator on $\mc{H}(E^\nu)$ with norm $\mc{D}$. Then:
\begin{enumerate}
\item For all real $x$
\begin{align}\label{phase-upperbound}
\varphi'(x) \le \mc{D} \sqrt{\nu};
\end{align}
\item  The zeros  of $E$ are separated from the real axis, that is, if $E(x-iy)=0$ for $x,y\in\R$ then
\begin{align}\label{Ezero-lowerbound}
y \geq \frac{1}{\D\sqrt{\nu}}.
\end{align}
\end{enumerate}
\end{proposition}

\begin{proof} Item (1).  Using the reproducing kernel definition \eqref{Intro_Def_K_E} we deduce that 
$$
K_\nu(x,x) =  \nu |E(x)|^{2(\nu-1)} K(x,x),
$$
for all real $x$, where $K_\nu(w,z)$ and $K(w,z)$ are respectively the reproducing kernels associated with $\H(E^\nu)$ and $\H(E)$. This implies that 
\es{\label{K-nu-kernel-id}
K_\nu(t,t) = \frac{\nu}{\pi} A(t)^{2\nu-1} B'(t),
}
for all $t\in \mc{T}_B$. We prove \eqref{phase-upperbound} first for  $x=t \in \mc{T}_B$. Consider the entire function $F$ defined by
\[
F(z) = E(z)^{\nu-2}B(z)K(t,z).
\]

Evidently, $F\in \mc{H}(E^\nu)$ and by \eqref{CS-ineq} we obtain
\[
|F'(t)|^2 \le \|F'\|^2_{E^\nu} K_\nu(t,t) \le \D^2 \|F\|^2_{E^\nu} K_\nu(t,t).
\]

A direct calculation gives
$$
\|F\|^2_{E^\nu} \leq \|K(t,z)\|^2_{E} = B'(t)A(t) /\pi
$$
and
$F'(t) = A(t)^{\nu-1} B'(t)^2/\pi$. Using identity \eqref{phi-id} in the form $\varphi'(t) = B'(t)/A(t)$ for $t\in \mc{T}_B$ implies \eqref{phase-upperbound} for  $t \in \mc{T}_B$.

Now, let $\theta\in\R$ be arbitrary and denote by $\varphi_\theta(x)$ the phase of $E_\theta(z) = e^{i\theta} E(z) = A_\theta(z)-iB_\theta(z)$. Observe that $\varphi_\theta(x)=\p(x)+\theta$ for all real $x$ and $E^\nu$ and $E_\theta^\nu$ generate the same space. Thus $E_\theta$ does not have real zeros and the real zeros of $B_\theta$ coincide with the points $\p(x)\equiv \theta \mod \pi$. Hence, the above argument for the space $\mc{H}(E_\theta^\nu)$ gives the claim for arbitrary $x\in\R$.

\noindent Item (2).  By \cite[Proposition 1.2]{Bar2} the function $E$ is of exponential type. Moreover, it cannot have a real zero $z_0$, since every function $F\in\H(E^\nu)$ would vanish at $z_0$. This is not possible since $F'\in\H(E^\nu)$ for all $F\in\H(E^\nu)$.

We conclude that $E$ has the following form (see \cite[Theorem 7.8.3]{Bo})
\begin{align*}
E(z) = Ce^{-iaz}\prod \bigg(1-\frac{z}{\ov w_n}\bigg)e^{ z\re (\frac{1}{w_n})},
\end{align*}
where $C$ is a constant, $\re a \geq 0$ and $\ov w_n=x_n-iy_n$ are the zeros of $E$ which satisfy
$$
\sum_{n} \frac{1+y_n}{|w_n|^2} < \infty.
$$
This allow us to deduce that the derivative of the phase function associated with $E$ is given by
$$
\p'(x) = \re a + \sum_{n} \frac{y_n}{(x-x_n)^2+y_n^2}.
$$
We deduce that $\p'(x_n)\geq y_n^{-1}$. The proof of this item is concluded once we use \eqref{phase-upperbound}.
\end{proof}

\begin{corollary}\label{ABcorollary}
   Let $\nu\ge 2$ be an integer. If differentiation defines a bounded operator on $\mc{H}(E^\nu)$, then  
 $A,B\notin \H(E)$.
\end{corollary}
\begin{proof}
  We can apply \cite[Corollary 2]{Bar3} to deduce that if $B\in \mc{H}(E)$ (or $A\in \mc{H}(E)$) then the zeros $x_n- iy_n$ of $E$ satisfy $\sum_n y_n<\infty$. This contradicts \eqref{Ezero-lowerbound}.
\end{proof}

\begin{proof}[Proof of Theorem \ref{Thm1}] If $E'/E\in H^\infty(\C^+)$, then \cite[Theorem 3.2]{Bar} applied to $E^\nu$ shows that differentiation defines a bounded operator on $\mc{H}(E^\nu)$ for all $\nu\geq 1$. 

If differentiation defines a bounded operator on $\mc{H}(E^\nu)$ for some $\nu\ge 2$, then \eqref{Ezero-lowerbound}  holds. It follows that $E'/E\in H^\infty(\C^+)$ by \cite[Theorem A]{Bar2}.
\end{proof}


\section{Frames for $\mc{H}(E^\nu)$}\label{derivative-frames}

As in the previous section $E$ is a Hermite-Biehler function with phase $\varphi$. Recall that $\mc{B}_\nu$ is the collection of functions
\[
B_{\nu,j}(z,t) = \frac{B(z)^\nu}{(z-t)^j}
\]
where $t\in \mc{T}_B$ and $1\le j\le \nu$, and $\mc{G}_\nu$ is the collection of functions $G_{\nu,j}(z,t)$  defined in \eqref{intro-Gdef}. The recipe for the proof of Theorem \ref{Thm2} is the following:
\begin{enumerate}
\item We show that the span of the collection $\G_\nu$ is dense in $\H(E^\nu)$, which by \eqref{G-delta-property} implies that there exists a dense set of functions in $\H(E^\nu)$  for which \eqref{intro-Gdual} holds;

\item  We derive estimates involving the inner products of the collection $\G_\nu$ in order to prove   Theorem
\ref{Thm2} item (2) for a dense set of functions (and hence for the whole space). 
\end{enumerate}

\begin{lemma}\label{BGdense} Assume $E$ has no real zeros, $B\not\in\H(E)$, and $\nu\geq 1$. Then the span of $\mc{B}_\nu$ and the span of $\G_\nu$ defined in \eqref{intro-Bdef} and \eqref{intro-Gdef} are both dense in $\mc{H}(E^\nu)$. 
\end{lemma}

\begin{proof} First we show via induction on $\nu$ that the span of the collection $\B_\nu$ is dense in $\H(E^\nu)$. It follows from \cite[Theorem 22]{dB} that the span of $\mc{B}_1$ is dense in $\H(E)$. Let $\nu\geq 1$ and assume that the span of $\mc{B}_\nu$ is dense in $\H(E^\nu)$. It follows from \cite[Lemma 4.1]{Bar} that if $E_a$ and $E_b$ are two Hermite-Biehler functions then
$$
\H(E_aE_b) = E^*_a\H(E_b)\oplus E_b\H(E_a)
$$ 
where the sum is orthogonal. This implies that
\est{
\H(E^{\nu+1}) = A \H(E^\nu) \oplus B \H(E^\nu),
}
where the sum is direct but possibly non-orthogonal. Thus, by induction assumption the span of the collection $\mc{C} = A\mc{B}_\nu \cup B \mc{B}_\nu$ is dense in $\H(E^{\nu+1})$. Evidently $B\mc{B}_\nu$ is a subset of $\mc{B}_{\nu+1}$. It remains to show that the collection $A\mc{B}_\nu$ can be arbitrarily approximated in the norm of $\H(E^{\nu+1})$ by elements of the span of $\mc{B}_{\nu+1}$.  We start by showing that $AB_{\nu,\nu}$ is contained in the closure of the span of $\B_{\nu+1}$ in $\H(E^{\nu+1})$.

The identity 
\begin{align*}
A(z) B_{\nu,\nu}(z,t) =\frac{A(t)}{B'(t)}[ B_{\nu+1,\nu+1}(z,t) - B(z) C(z,t)]
\end{align*}
holds for $t\in\mc{T}_B$, where
\begin{align*}
C(z,t) &= B_{\nu-1,\nu-1}(z,t) \frac{A(t) B(z) - B'(t)(z-t) A(z)}{A(t)(z-t)^2}
\end{align*}
is an element of $ \H(E^\nu)$. The induction assumption in conjunction with inequality  $\|BG\|_{E^{\nu+1}}\le \|G\|_{E^\nu}$ for $G\in \mc{H}(E^\nu)$  may be used to show that $BC$ is contained in the span of $B \mc{B}_{\nu} \subseteq \mc{B}_{\nu+1}$ in $\mc{H}(E^{\nu+1})$. It follows that $AB_{\nu,\nu}$ is contained in the closure of the span of $\mc{B}_{\nu+1}$ in $\mc{H}(E^{\nu+1})$. If $1\leq j\leq \nu-1$ then evidently $A(z) B_{\nu,j}(z,t) = B(z) H(z)$ for some $H\in \H(E^\nu)$ and the same argument may be applied. This proves the first part of the lemma.

Now, by \eqref{1/B-P-def} and \eqref{intro-Gdef} we deduce that
\es{\label{G-B-relation}
G_{\nu,j}(z,t) &= \frac{1}{j!} \sum_{n=0}^{\nu-j-1} a_{\nu,n}(t) B_{\nu,\nu-j-n}(z,t)\\
&= \frac{1}{j!} \sum_{n=1}^{\nu-j} a_{\nu,\nu-j-n}(t) B_{\nu,n}(z,t).
}
Suppressing the arguments $t$ and $z$, this is in matrix notation
\es{\label{matrix-eq}
\left[\begin{matrix}
G_{\nu,0} \\  G_{\nu,1} \\ \vdots \\ G_{\nu,\nu-1}
\end{matrix}\right]
=
\left(\begin{matrix}
\frac{a_{\nu,\nu-1}}{0!} & \hdots & \frac{a_{\nu,1}}{0!}  & \frac{a_{\nu,0}}{0!} \\
\frac{a_{\nu,\nu-2}}{1!} & \hdots & \frac{a_{\nu,0}}{1!}  &  0 \\
\vdots & & 0 & 0 \\
\frac{a_{\nu,0} }{(\nu-1)!} & \hdots & 0 & 0 
\end{matrix}\right)
\left[\begin{matrix}
B_{\nu,1} \\  B_{\nu,2} \\ \vdots \\ B_{\nu,\nu}
\end{matrix}\right].
}
Since $a_{\nu,0}(t) = 1/B'(t)^\nu\neq 0$, it follows that the above matrix is invertible and, in particular, any element of $\mc{B}_\nu$ is a linear combination of elements from $\mc{G}_\nu$ and vice versa, which finishes the proof.
\end{proof}


\begin{lemma} Let $\nu\ge 2$ be an integer. Assume that differentiation defines a bounded operator on $\mc{H}(E^\nu)$ with norm $\mc{D}$. Then:
\begin{enumerate}
\item The zeros of $B$ are separated, that is, we have 
\begin{align}\label{zero-separation}
|t-s|\ge \frac{\pi}{\D\sqrt{\nu}},
\end{align}
for all $s,t\in\Tau_B$;

\item For all real $x$ and $t\in \mc{T}_B$ we have
\begin{align}\label{Bt-bound}
\left|\frac{B(x)}{E(x)(x-t)}\right|\leq \D\sqrt{\nu}.
\end{align}

\end{enumerate}
\end{lemma}

\begin{proof} The positivity of $\varphi'$ and the bound $\varphi'(x) \le D\sqrt{\nu}$ from \eqref{phase-upperbound} imply
\[
 |\p(s) - \p(t)| \leq \D\sqrt{\nu} |s-t|,
\]
for any $s,t\in\R$. If $s,t\in\Tau_B$ are consecutive elements, then the left side equals $\pi$. 

Inequality \eqref{CS-ineq} gives $|K(w,z)|^2\leq K(w,w)K(z,z)$ for all $w,z\in\C$. We obtain
\[
\bigg|\frac{B(x)}{E(x)(x-t)}\bigg|^2 = \pi^2 \frac{K(t,x)^2}{A(t)^2 |E(x)|^2}\leq \pi^2\frac{K(t,t)}{A(t)^2}\frac{K(x,x)}{|E(x)|^2}=\p'(t)\p'(x) \leq \D^2\nu
\]
which finishes the proof.
\end{proof}

\begin{lemma}\label{lemma-norm-mult-B}
Assume $A\notin \H(E)$. Then for any integer $\nu\geq 2$ and any $F\in\H(E)$ we have
\es{\label{norm-equiv-mult-B}
{\nu}^{-1/2}\|F\|_{E} \leq \|B^{\nu-1}F\|_{E^\nu} \leq \|F\|_{E}.
}
In particular, we have that
$$
\bigg\|\frac{B(z)^\nu}{(z-t)}\bigg\|^2_{E^\nu} \simeq \p'(t),
$$
for all $t\in\Tau_B$.
\end{lemma}

\begin{proof} The right hand side inequality in \eqref{norm-equiv-mult-B} follows from $|B(x)| \leq |E(x)|$ for real $x$.
For every $\nu\geq 2$ let $E(z)^\nu = A_\nu(z)-iB_\nu(z)$, where $A_\nu(z)$ and $B_\nu(z)$ are real entire functions. It is simple to see that when $\nu$ is even $B_\nu/A $ is entire and when $\nu$ is odd $A_\nu/ A$ is entire. In either case, the zeros of $A$ form a subset of the zeros that occur in \eqref{norm-HE} for $\mc{H}(E^\nu)$, and we obtain 
\est{
\|B^{\nu-1}F\|_{E^\nu}^2 & \geq \sum_{A(s)=0} \frac{|B(s)|^{2\nu-2}|F(s)|^2}{K_\nu(s,s)} = \frac{\pi}{\nu}\sum_{A(s)=0}\frac{|B(s)|^{2\nu-2}|F(s)|^2}{|A'(s)B(s)^{2\nu-1}|}
\\ & = \nu^{-1}\sum_{A(s)=0} \frac{|F(s)|^2}{K(s,s)} = \nu^{-1}\|F\|^2_E
}
with another application of \eqref{norm-HE}.
\end{proof}

\begin{lemma} \label{zero-inner-prod}
Let $\nu\geq 1$ be an integer. Then for all distinct $s, t\in \mc{T}_B$ we have
\[
\langle B_{\nu,1}(\cdot,s) , B_{\nu,1}(\cdot,t)\rangle_{E^\nu} =0.
\]
\end{lemma}

\begin{proof}
We define an entire function $I = I_{s,t}$ by
\[
I(z) = \frac{B(z)^{2\nu}}{(E(z)(E^*(z))^\nu(z-s)(z-t)}
\]
where $s$ and $t$ are two zeros of $B$. The decomposition
\begin{align*}
\frac{B(z)^{2\nu}}{(E(z)E^*(z))^\nu}  = \frac{i^\nu}{2^\nu} \sum_{j=0}^{2\nu} \binom{2\nu}{j}\frac{ E(z)^{j} E^*(z)^{2\nu-j}}{(E(z)E^*(z))^\nu} (-1)^{2\nu-j}\\
 =\frac{i^\nu}{2^\nu} \sum_{j=0}^{2\nu} \binom{2\nu}{j} \left( \frac{ E(z)}{E^*(z)} \right)^{j-\nu}(-1)^{2\nu-j}.
\end{align*}
suggests to consider $I_j$ for $j\in\{0,...,2\nu\}$ defined by
\[
I_j(z) =  \frac{1}{(z-s)(z-t)}\left( \frac{ E(z)}{E^*(z)} \right)^{j-\nu}.
\]

Evidently each $I_j$ is meromorphic with poles at $s$ and $t$, and additional poles in the upper half plane if $j>\nu$ and the lower half plane if $j<\nu$. This suggest to consider a contour $C_K$ consisting of a deformation of $[-K,K]$ by small semicircles avoiding $s$ and $t$, and closed by a large semicircle in the appropriate half-plane of radius $K$ and center at the origin. Since $E^*/E$ is bounded by $1$ in the upper half plane (and $E/E^*$ analogously in the lower half plane), a standard residue theorem argument and the identities $E(t) = E^*(t)$ for $t\in\mc{T}_B$  may be used to show
\[
\frac{1}{2\pi i} \int_{C_K} I_j(z) \dz =  0,
\]
the details are left to the reader.
\end{proof}

\begin{lemma} \label{est-inner-prod-norm-B}  Assume differentiation defines a bounded operator on $\H(E^\nu)$ with norm $\mc{D}$. Then the following statements hold:

\begin{enumerate}
\item For any distinct $s,t\in\mc{T}_B$ and $j= 2,...,\nu$ we have
\[
|\langle B_{\nu,j}(\cdot,s),B_{\nu,j}(\cdot,t) \rangle_{E^\nu}| \lesssim \frac{1}{(s-t)^2};
\]
\item For any $t\in \mc{T}_B$ and $j=1,...,\nu$  we have
\[
\|B_{\nu,j}(\cdot,t)\|_{E^\nu} \lesssim 1;
\]
\item Denote by $a_{\nu,j}(t)$ the coefficient of $(z-t)^j$ in the Taylor series expansion of $B_{\nu,\nu}(z,t)^{-1}$ about the point $z=t$. Assume also that $\delta=\inf_{t\in\Tau_B} \p'(t)$. Then
\begin{align*}
|a_{\nu,j}(t)|^2\lesssim \frac{1}{K_\nu(t,t)},
\end{align*}
for all $t\in\Tau_B$ and $j=0,...,\nu-1$.
\end{enumerate}
All the implied constants above depend only on $\nu, \D$ and $\delta$.
\end{lemma}

\begin{proof} 
{ Item} (1). Using the fact that $|B(x)|\leq |E(x)|$ for all real $x$ we deduce that
\est{
\langle B_{2,2}(z, s),B_{2,2}(z, t)\rangle_{E^{2}} \leq \bigg\| \frac{B(z)}{(z-t)(z-s)}\bigg\|_{E}^2 & = \pi\bigg(\frac{\p'(t)+\p'(s)}{(t-s)^2}\bigg)\\ 
& \lesssim (t-s)^{-2},
}
where the identity above is due to formula \eqref{norm-HE} and the last inequality due to \eqref{phase-upperbound}. Now, for any $j=2,\ldots,\nu$ we have the following sequence of estimates
\est{
\bigg|\frac{B_{\nu,j}(x, t)B_{\nu,j}(x, s)}{|E(x)|^{2\nu}}\bigg| & =\frac{B(x)^{2\nu}}{|x-t|^j|x-s|^j|E(x)|^{2\nu}} \\
& \leq \frac{B(x)^{2j}}{|x-t|^j|x-s|^j|E(x)|^{2j}} \\
& \lesssim \frac{B(x)^{4}}{|x-t|^2|x-s|^2|E(x)|^{4}},
}
where the last inequality above is due to \eqref{Bt-bound}. We conclude that
\est{
|\langle B_{\nu,j}(z, s),B_{\nu,j}(z, t)\rangle_{E^{\nu}}| \lesssim \langle B_{2,2}(z, s),B_{2,2}(z, t)\rangle_{E^{2}} \lesssim (t-s)^{-2}.
}

{ Item} (2). We apply  \eqref{Bt-bound} and $|B/E|\le 1$ on $\R$ to deduce that
\[
\left|\frac{B(x)^\nu}{(x-t)^{j} E(x)^{\nu} }\right|  \lesssim \left|\frac{B(x)}{ (x-t) E(x)}\right|,
\]
for all real $x$ and $j=1,\ldots,\nu$. It follows that
\[
\| B_{\nu,j}(z, t)\|^2_{E^{\nu}} \lesssim \bigg\|\frac{B(z)}{z-t}\bigg\|^2_E =\pi\p'(t) \leq \pi \D\sqrt{\nu}.
\]

{ Item} (3). Denote by $b_{\nu,j}(t)$ the coefficients of the power series expansion of $B_{\nu,\nu}(z, t)$ as a function of $z$ about $z=t$. The assumption that $\varphi'(t)\geq \delta$ whenever $B(t)=0$ in conjunction with identity \eqref{K-nu-kernel-id} and estimate \eqref{phase-upperbound} implies that
$$
|b_{\nu,0}(t)|^2 = |B'(t)^\nu|^2 \simeq K_\nu(t, t).
$$
Also, for $j=1,\ldots,\nu$ we have
\begin{align*}
|b_{\nu,j}(t)|^2= \frac{1}{(j!)^2}|B_{\nu,\nu}^{(j)}(t, t)|^2 \leq  \frac{1}{(j!)^2}\| B_{\nu,\nu}^{(j)}(\cdot, t)\|^2_{E^\nu} K_\nu(t, t) \lesssim K_\nu(t, t).
\end{align*}
We obtain
\es{\label{ineq-10}
\frac{|b_{\nu,j}(t)|}{|b_{\nu,0}(t)|} \lesssim 1.
}

Now note that for $\ell=1,\ldots,\nu-1$
$$
0=\partial_z^\ell \bigg[\frac{B_{\nu,\nu}(z, t)}{B_{\nu,\nu}(z, t)}\bigg]_{z=t} = \ell!\sum_{j=0}^\ell a_{\nu,j}(t)b_{\nu,\ell-j}(t).
$$
Hence the relation between $a_{\nu,j}(t)$ and $b_{\nu,j}(t)$ is given by a triangular matrix with diagonal terms equal to $b_{\nu,0}(t)$. Using \eqref{ineq-10} we conclude that
$$
|a_{\nu,j}(t)|^2 \lesssim 1/K_\nu(t, t).
$$
This concludes the lemma.
\end{proof}

For the sake of completeness we state here a result about Hilbert-type inequalities \cite[Corollary 22]{CLV}.

\begin{lemma}
Let $\lambda_1, \lambda_2,\ldots, \lambda_N$ be real numbers such that $|\lambda_n-\lambda_m|\geq \sigma$ whenever $m \neq n$, for some $\sigma>0$. Let $a_1, a_2,\ldots, a_N$ be complex numbers. Then
\begin{align}\label{Hilbert-inequality}
-\frac{\pi^2}{6\sigma^2}\sum_{n=1}^N |a_n|^2 \leq \sum_{\stackrel{m,n=1}{m\neq n}}^N\frac{a_n\ov{a_m}}{(\lambda_n-\lambda_m)^2}\leq \frac{\pi^2}{3\sigma^2}\sum_{n=1}^N|a_n|^2.
\end{align}
The constants appearing in these inequalities are the best possible $($as $N\to\infty)$.
\end{lemma}

The next lemma estimates the norm of the linear combination of elements from $\G_\nu$. This is one of the two inequalities needed to show that this collection is a (weighted) frame.

\begin{lemma}\label{frame-G} 
Assume $\varphi'(t)\ge \delta>0$ for all $t\in \mc{T}_B$, and assume differentiation defines a bounded operator on $\mc{H}(E^\nu)$ with norm $\mc{D}$. Let $c_j(t)\in \C$ for $t\in \mc{T}_B$ and $j\in\{0,...,\nu-1\}$ be such that 
\[
\sum_{t\in \mc{T}_B} \sum_{j=0}^{\nu-1} \frac{|c_j(t)|^2}{K_\nu(t,t)}<\infty.
\]
Then the series
$$
C(z)=\sum_{t\in \mc{T}_B} \sum_{j=0}^{\nu-1} c_j(t) G_{\nu,j}(z,t)
$$
converges in the norm of $\H(E^\nu)$ and we have
\begin{align*}
\| C\|^2_{\mc{H}(E^\nu)} \lesssim  \sum_{t\in \mc{T}_B} \sum_{j=0}^{\nu-1} \frac{|c_j(t)|^2}{K_\nu(t,t)},
\end{align*}
where the implies constant depends only on $\nu,\D$ and $\delta$.
\end{lemma}

\begin{proof} 
Define
\[
d_m(t) =  \sum_{j=0}^{\nu-m}  a_{\nu,\nu-m-j}(t) \frac{c_j(t)}{j!}.
\]
By item $(3)$ of Lemma \ref{est-inner-prod-norm-B} we trivially obtain
\es{\label{ineq-13}
\sum_{t\in\Tau_B}\sum_{m=1}^{\nu}|d_m(t)|^2 \lesssim \sum_{t\in\Tau_B}\sum_{j=0}^{\nu-1}\frac{|c_j(t)|^2}{K_\nu(t,t)} <\infty.
}
Let $S\subset \Tau_B$ be any finite set. Using \eqref{G-B-relation} we obtain
\begin{align*}
\Big\|\sum_{t\in S}  \sum_{j=0}^{\nu-1} c_j(t) G_{\nu,j}(z,t)\Big\|_{E^\nu}^2 & = \Big\|\sum_{t\in S}  \sum_{j=0}^{\nu-1} \frac{c_j(t)}{j!} \sum_{m=1}^{\nu-j} a_{\nu,\nu-m-j}(t) B_{\nu,m}(z,t)\Big\|_{E^\nu}^2 \\
&= \Big\|\sum_{t\in S}  \sum_{m=1}^{\nu} d_m(t)  B_{\nu,m}(z,t)\Big\|_{E^\nu}^2\\
&\lesssim \sum_{m=1}^{\nu} \sum_{t\in S} |d_m(t)|^2 +  \sum_{m=2}^{\nu}\sum_{\substack{s,t\in S \\ s\neq t }} \frac{|d_m(t)||d_m(s)|}{(t-s)^2} \\
&\lesssim  \sum_{m=1}^{\nu}  \sum_{t\in S} |d_m(t)|^2 \\
& \lesssim  \sum_{t\in S} \sum_{j=0}^{\nu-1} \frac{|c_j(t)|^2}{K_\nu(t,t)},
\end{align*}
where the first inequality is due to Lemmas \ref{zero-inner-prod} and \ref{est-inner-prod-norm-B}, and the third inequality follows from \eqref{ineq-13}. The second term on the right hand side of the third line in the above calculation is in the form of a Hilbert-Type inequality. By \eqref{zero-separation} the points $\Tau_B$ are uniformly separated, hence the second inequality above follows by \eqref{Hilbert-inequality}. 
Since the constants do not depend on the subset $S$, the lemma follows.
\end{proof}


\subsection{Proof of Theorem \ref{Thm2}}
$\nonumber$

We aim to show first that for $F\in\H(E^\nu)$ the series in \eqref{intro-Gdual} converges to $F$ in $\mc{H}(E^\nu)$. Since $E'/E \in H^\infty(\C^+)$, we obtain from Theorem \ref{Thm1} that  $F^{(j)} \in \mc{H}(E^\nu)$ for every $j\geq0$. Since $\Tau_{B}\subset \Tau_{B_\nu}$, 
we may apply \eqref{norm-HE} in $\mc{H}(E^\nu)$ to $F^{(j)}$ to obtain
\begin{align}\label{ineq-14}
\sum_{j=0}^{\nu-1}\sum_{t\in \Tau_B} \frac{|F^{(j)}(t)|^2}{K_\nu(t,t)} \lesssim \|F\|^2_{E^\nu}.
\end{align}

Thus, we may apply Lemma \ref{frame-G} to deduce that the interpolation formula \eqref{intro-Gdual} converges in $\mc{H}(E^\nu)$ to a function $F_0$. We claim that $F_0=F$. By Lemma \ref{BGdense}, the span of $\G_\nu$ is dense in $\H(E^\nu)$ and by property \eqref{G-delta-property} any function $H$ in the span of $\G_\nu$ satisfies interpolation formula \eqref{intro-Gdual}. Evidently $F_0$ is in the closure of the span of $\mc{G}_\nu$, hence we have
\[
\|H-F_0\|^2_{E^\nu} \lesssim  \sum_{j=0}^{\nu-1}\sum_{t\in \Tau_B} \frac{|H^{(j)}(t)-F^{(j)}(t)|^2}{K_\nu(t,t)},
\]
for all $H\in \G_\nu$.
An application of \eqref{ineq-14} bounds the sum on the right in terms of $\|H - F\|_{E^\nu}$. Adding and subtracting suitable $H$ in $\|F - F_0\|_{E^\nu}$ shows that this norm equals zero, that is, $F = F_0$.

 We prove next relation \eqref{intro-Gframe}. Define 
\[
\Delta_{\nu,j,t}(z) = K_\nu(t,t)^{-\frac12} D_{\nu,j}(z,t). 
\]

Equation \eqref{intro-Dframe} implies that $\{ \Delta_{\nu,j,t}: j=0,...,\nu-1; t\in\mc{T}_B\}$ is an (unweighted) frame for $\H(E^\nu)$. Consider the frame operator $U:\H(E^\nu)\to \H(E^\nu)$ defined by
\[
UF(z) = \sum_{t\in\mc{T}_B} \sum_{j=0}^{\nu-1} \langle F, \Delta_{\nu,j,t}\rangle_{\H(E^\nu)} \Delta_{\nu,j,t}(z).
\]

It is a basic result of frame theory  (see \cite[Corollary 5.1.3]{G}) that $U$ is invertible and positive, and that the collection $\{ U^{-1} \Delta_{\nu,j,t} :  j=0,...,\nu-1;t\in \mc{T}_B\}$ is also a frame, sometimes called the canonical dual frame. It follows immediately that
\[
U G_{\nu,j} = K_\nu(t,t)^{-\frac12} \Delta_{\nu,j,t}(z),
\]
which implies that $\{ K_\nu(t,t)^{\frac12}G_{\nu,j}(.,t) : j=0,...,\nu-1; t\in\mc{T}_B\}$ is a dual frame of $\mc{\D}_\nu$. This implies \eqref{intro-Gframe}. We remark that since for every fixed $t\in \mc{T}_B$ the functions $G_{\nu,j}(z,t)$ and $B_{\nu,j}(z,t)$ are connected via an invertible matrix transformation, the inequalities can also be shown from the bounds for $\mc{B}_\nu$ established in Lemma \ref{est-inner-prod-norm-B}. Finally, the fact that the relations \eqref{intro-Dframe} and \eqref{intro-Gframe} fail to hold if a term is removed  is a direct consequence of property \eqref{G-delta-property} and relation \eqref{matrix-eq}. The proof of Theorem \ref{Thm2} is complete.

\subsection{Proof of Theorem \ref{Thm3}}

Let $\nu\ge 2$. We assume that $E'/E\in H^\infty(\C^+)$ and
\begin{align}\label{lower-Dframe}
\|F\|_{E^\nu} \lesssim \sum_{t\in\mc{T}_B} \sum_{j=0}^{\nu-1} \frac{|F^{(j)}(t)|^2}{K_\nu(t,t)}
\end{align}
for all $F\in \mc{H}(E^\nu)$.   Using Corollary \ref{ABcorollary} in conjunction with Lemma \ref{lemma-norm-mult-B} we deduce that
$$
\bigg\|\frac{B(z)^\nu}{z-t}\bigg\|_{E^\nu}^2 \simeq \p'(t),
$$
for any $t\in\Tau_B$. From \eqref{lower-Dframe} and identity \eqref{K-nu-kernel-id} we deduce that
$$
\bigg\|\frac{B(z)^\nu}{(z-t)}\bigg\|_{E^\nu}^2 \lesssim \frac{B'(t)^{2\nu}((\nu-1)!)^2}{K_\nu(t,t)} = \frac{\pi((\nu-1)!)^2}{\nu} \p'(t)^{2\nu-1}.
$$

Since the implicit constants are independent of $t$, division by $\varphi'(t)$ gives  $1\lesssim \varphi'(t)$ for all $t\in \mc{T}_B$.


\section{The \texorpdfstring{$L^p$}{Lp} case}\label{Lp-case}

De Branges spaces are closely related to Hardy spaces in the upper half-plane $\C^+$. For a given $p\in[1,\infty]$, the Hardy space $H^p(\C^+)$ is defined as the space of holomorphic functions $F:\C^+\to\C$ such that
\est{ 
\sup_{y>0} \| F(x+iy)\|_{L^p}<\infty,
}
where $\|\cdot\|_{L^p}$ denotes the standard $L^p$-norm in the variable $x$-variable. In this situation, the limit $F(x)=\lim_{y\to 0} F(x+iy)$ exists in the $L^p$-sense and 
\est{\label{norm-Hp}
\sup_{y>0} \| F(x+iy)\|_{L^p} = \|F(x)\|_{L^p}.
}
This space endowed with the norm \eqref{norm-Hp} defines a Banach space of holomorphic functions on the upper half-plane.

Given a Hermite-Biehler function $E$ one can prove that an entire function $F$ belongs to $\H(E)$ if and only if $F/E$ and $F^*/E$ belong to the space $H^2(\C^+)$. This equivalent definition allows us to define de Branges spaces for any given exponent $p\in[1,\infty]$. The $L^p$ de Branges space $\H^p(E)$ is defined as the space of entire functions $F$ such that $F/E$ and $F/E^*$ belong to $H^p(\C^+)$ (hence $\H^2(E)=\H(E)$). Using the fact that $H^p(\C^+)$ is a Banach space one can easily prove that $\H^p(E)$ is  a Banach space of entire functions with norm given by
$$
\|F\|_{E,p} = \bigg(\int_\R \bigg|\frac{F(x)}{E(x)}\bigg|^p\dx\bigg)^{1/p}
$$
for finite $p$, or
$$
\|F\|_{E,\infty} = \sup_{x\in\R}\bigg|\frac{F(x)}{E(x)}\bigg| 
$$
for $p=\infty$. For all these facts see \cite[Section 3]{Gon} and \cite{Bar,Bar2}.

Evidently $\| K(w,.)\|_{E,q}<\infty$ for every $1<q\le \infty$ and $w\in\C$. It follows that these spaces have a reproducing kernel structure given by the Cauchy formula for functions of Hardy spaces (see \cite[Theorems 13.2 and 13.5]{Ma}), namely if $p\in [1,\infty)$ and $F\in\H^p(E)$ then
\es{\label{rep-ker-Lp}
F(w)=\int_\R \frac{F(x)\ov{K(w,x)}}{|E(x)|^2}\dx
}
for every $w\in\C$, where $K(w,z)$ is defined in \eqref{Intro_Def_K_E}. Using H\"older's inequality we obtain an important estimate
\es{\label{holder-Lp}
|F(w)|\leq \|F\|_{E,p}\|K(w,\cdot)\|_{E,p'},
}
where $p'$ is the conjugate exponent of $p$.  Using the known fact that the space  $H^{p'}(\C^+)$ can be identified with the dual space of $H^p(\C^+)$ for $p\in(1,\infty)$ one can deduce that $\H^p(E)'=\H^{p'}(E)$ if $p\in(1,\infty)$, that is, if $\Lambda$ is a functional over $\H^p(E)$ then there exists a function $\Lambda\in\H^{p'}(E)$ such that
$$
\langle \Lambda,\, F\rangle = \int_\R \frac{F(x)\ov{\Lambda(x)}}{|E(x)|^2}\dx,
$$
for all $F\in\H^p(E)$. The proof of this duality result deals with model spaces which diverges from the purposes of this article. For the interested reader we refer to \cite[Proposition 1.1]{Bar2} and \cite[Lemma 4.2]{Co}.

It is the goal of this section to prove that the interpolation series of $F\in \mc{H}^p(E^\nu)$ represents $F$. We conjecture that convergence of the series takes place in $\mc{H}^p(E^\nu)$ as well, but we do not have a proof of this statement.

\begin{theorem}\label{Lp-recont-thm}
Let $\nu\geq 2$ be an integer and let $1\le p<\infty$. Let $E$ be a Hermite-Biehler function with $E'/E\in H^\infty(C^+)$, and assume that there exists $\delta>0$ so that $\varphi'(t)\ge \delta$ for all $t\in\Tau_B$. Then
\es{\label{Lp-recont-form}
F(z) = \sum_{t\in \mc{T}_B} \sum_{j=0}^{\nu-1} F^{(j)}(t) G_{\nu,j}(z,t)
}
where the formula converges uniformly in compact subsets of $\C$.
\end{theorem}

\subsection{Preliminaries} We collect facts about $H^p(E^\nu)$ that will be needed in the proof of Theorem \ref{Lp-recont-thm}.

\begin{lemma}\label{inclusion-lemma}
Assume that the phase $\p$ of $E$ has bounded derivative on $\R$. Then for $1\leq p < q <\infty$ we have $\H^p(E)\subset \H^q(E)$ continuously. Also, if $p>1$ then $\H^p(E)$ is dense in $\H^q(E)$.
\end{lemma}

\begin{proof}
The inclusion part is \cite[Lemma 9]{Gon}. The second part is an straightforward application of the Hahn-Banach Theorem and the reproducing kernel property \eqref{rep-ker-Lp}.
\end{proof}

The following proposition collects relations between the condition $\frac{E'}{E} \in H^\infty(\C^+)$ and boundedness of the differentiation operator on $\mc{H}^p(E^\nu)$. It is not clear if boundedness of differentiation on $\mc{H}^p(E^\nu)$ for $\nu\ge 2$ and $E'/E\in H^\infty(\C^+)$ are equivalent. We are not able to prove it for $p\neq 2$, but we are also not aware of any counterexample. 

We say that the zeros $\ov w_n=x_n-iy_n$ of $E$ are separated from the real line if $\inf_n y_n>0$.

\begin{proposition}\label{equiv-con-Lp-prop}
Let $\nu\geq 2$ and $p\in (1,\infty)$.
\begin{enumerate}
\item If  $E'/E\in H^\infty(\C^+)$ then $\H^p(E^\nu)$ is closed under differentiation and the zeros of $E$ are separated from the real line.
\item Assume $\H^p(E^\nu)$ is closed under differentiation. Then $E'/E\in H^\infty(\C^+)$ if and only if the zeros of $E$ are separated from the real line.
\item If $\H^p(E^\nu)$ is closed under differentiation and $v(E^*/E)<0$ then $E'/E\in H^\infty(\C^+)$.
\end{enumerate}
\end{proposition}

\begin{proof}
{ Item} $(1)$. If $E'/E\in H^\infty(\C^+)$ then by Theorem \ref{Thm1} the space $\H^2(E^\nu)$ is closed under differentiation. By Proposition \ref{upper-phase-bound} the zeros of $E$ are separated from the real line, and  by \cite[Proposition 1.2]{Bar2} the function $E$ is of exponential type. We can now apply \cite[Theorem A]{Bar2} to deduce that $\H^p(E^\nu)$ is closed under differentiation.

{ Item} $(2)$. Assume $\H^p(E^\nu)$ is closed under differentiation. If $E'/E\in H^\infty(\C^+)$ then by Theorem \ref{Thm1} the space $\H^2(E^\nu)$ is closed under differentiation.  By Lemma \ref{upper-phase-bound} the zeros of $E$ are separated from the real line. Conversely, if the zeros of $E$ are separated from the real line then  \cite[Theorem A]{Bar2} implies $E'/E\in H^\infty(\C^+)$.

{ Item} $(3)$. This is a direct consequence of \cite[Theorem A]{Bar2}.
\end{proof}

We say that a sequence of real numbers $\{\lambda_n\}$ is $\ep$-separated, for some $\ep>0$, if $|\lambda_n-\lambda_n|\geq \ep$ for every $m\neq n$. We now prove a generalized (weighted) version of the P\'olya-Plancherel theorem (see \cite{PP}). 

\begin{lemma}\label{PP-lemma}
Let $E$ be a Hermite-Biehler function with zeros $\ov w_n = x_n-iy_n$ such that $h=\inf_n y_n >0$ and $\tau =\sup_x \varphi'(x)<\infty$. If  $\{\lambda_n\}$ is an $\ep$-separated sequence of real numbers, $p\in[1,\infty)$ and $F\in\H^p(E)$, then
$$
\sum_{n} \bigg|\frac{F(\lambda_n)}{E(\lambda_n)}\bigg|^{p} \leq \frac{1+e^{6\tau p \alpha}}{\pi\alpha}\int_{\R}\bigg|\frac{F(x)}{E(x)}\bigg|^p\dx
$$
where $\alpha=\min\{\ep/2,h/2\}$.
\end{lemma}
\begin{proof}  

{\it Step 1.} Since $E^*/E$ is bounded on the upper half-plane and has modulo one in the real line, by Nevanlinna's factorization (see \cite[Theorem 8]{dB}) we obtain
$$
\Theta(z):=\frac{E^*(z)}{E(z)} = e^{2aiz}\prod_n{\frac{1-z/w_n}{1-z/\ov w_n}}
$$
where $2a =-v(E^*/E)\geq 0$ ($v(\cdot)$ denotes the mean type of a function defined in \eqref{mean-type}). If $z=x+iy$ with $y\geq 0 $ we have the following identity
\est{
\frac{1}{2}\partial_y \log|\Theta^*(z)| = & a + \sum_n \frac{y_n[(x-x_n)^2 + y_n^2-y^2]}{|z-\ov w_n|^2|z-w_n|^2}.
}
If $0\leq y\leq h/2$ then $y_n^2-y^2\leq 3(y_n-y)^2$ and we deduce that
\est{
\frac{1}{2}\partial_y \log|\Theta^*(z)| \leq a + 3\sum_n \frac{y_n}{(x-x_n^2)+y_n^2} & \leq \frac{3}{2}\partial_y \log|\Theta^*(x)| \\ & = 3\p'(x)\leq 3\tau.
}
Integrating in $y$, we obtain $|\Theta^*(z)| \leq e^{6\tau y}$ for $0 \leq y \leq h/2$.

{\it Step 2.} Let $\alpha=\min\{h/2,\ep/2\}$. Since the function $|F(z)/E(z)|^p$ is sub-harmonic in the half-plane $\im z> -h$, its value at the center of a disk is not greater than its mean value over the disk. We obtain
\est{
|F(\lambda_n)/E(\lambda_n)|^p \leq &\frac{1}{\pi \alpha^2}\int_{0}^{\alpha} \int_{0}^{2\pi}|F(\lambda_n+\rho e^{it})/E(\lambda_n+\rho e^{it})|^p\dt\rho\d\rho \\
\leq  & \frac{1}{\pi \alpha^2} \int_{-\alpha}^{\alpha} \int_{\lambda_n-\alpha}^{\lambda_n+\alpha}|F(x+iy)/E(x+iy)|^p\dx\dy.
}
Using the separability of $\{\lambda_n\}$ we can sum the above inequality for all values of $n$ to obtain
\est{
\sum_n  |F(\lambda_n)/E(\lambda_n)|^p \leq & \frac{1}{\pi \alpha^2}\int_{-\alpha}^{\alpha} \int_{\R}|F(x+iy)/E(x+iy)|^p\dx\dy \\
 = \, & \frac{1}{\pi \alpha^2}\int_{0}^{\alpha} \int_{\R}|F(x+iy)/E(x+iy)|^p\dx\dy \\
  &  \,\,\,\,\,\,\,\,\,\, +\frac{1}{\pi \alpha^2}\int_{0}^{\alpha} \int_{\R}|F^*(x+iy)/E^*(x+iy)|^p\dx\dy.
}
Since $|\Theta^*(z)| \leq e^{6\tau y}$ for $0\leq y\leq h/2$ we conclude that
$$
\sum_n  |F(\lambda_n)/E(\lambda_n)|^p \leq \frac{1}{\pi \alpha^2}\int_{0}^{\alpha} \int_{\R}\frac{|F(x+iy)|^p+ e^{6\tau p \alpha}|F^*(x+iy)|^p}{|E(x+iy)|^p}\dx\dy.
$$
By definition, $F/E$ and $F^*/E$ belong to the Hardy space $H^p(\C^+)$. A basic property of Hardy spaces is that 
$$
\sup_{y>0}\|G(x+iy)\|_{L^p} = \lim_{y\to 0} \|G(x+iy)\|_{L^p}
$$
for every $G\in H^p(\C^+)$ (see \cite[Theorems 13.2 and 13.5]{Ma}). Using this last fact we obtain
$$
\int_{\R}\frac{|F(x+iy)|^p+ e^{6\tau p\alpha}|F^*(x+iy)|^p}{|E(x+iy)|^p}\dx \leq  (1+e^{6\tau p\alpha})\int_{\R}\bigg|\frac{F(x)}{E(x)}\bigg|^p\dx
$$
for every $y>0$. This concludes the lemma.
\end{proof}

\subsection{Proof of Theorem \ref{Lp-recont-thm}} 
We show first that the singular part of the function $F(z)/B(z)^\nu$ at a given zero $t\in\Tau_B$ is
$$
\sum_{j=0}^{\nu-1}F^{(j)}(t)\frac{G_{\nu,j}(z,t)}{B^\nu(z)} = \sum_{j=0}^{\nu-1} F^{(j)}(t)\frac{P_{\nu,\nu-j-1}(z,t)}{j!(z-t)^{\nu-j}}.
$$

To see this, define for any complex number $w$ a linear operator $\Sg_w$ on the space of meromorphic functions by
\est{
\Sg_w(G)(z)=\sum_{n\leq -1} g_n(z-w)^n
} 
if $G$ has the  series representation
$$
G(z)=\sum_{n\in\Z} g_n(z-w)^n
$$
about $z=w$. That is, $\Sg_w(G)$ is defined as the singular part of the function $G$ at the point $z=w$. Since $G$ is meromorphic, $\Sg_w(G)$ is always a rational function. It is a simple, but useful characterization that $S_w(G)$ is the unique rational function $R$ having exactly one pole which is located at $z=w$, such that $G(z)-R$ has a removable singularity at the point $z=w$ and 
\es{\label{eq-4}
\lim_{|z|\to\infty} R(z)=0.
}

Using \eqref{intro-Gdef} we obtain
\es{\label{eq-5}
\Sg_t\bigg(\frac{F}{B^\nu}\bigg)(z) = \sum_{j=0}^{\nu-1} F^{(j)}(t)\frac{P_{\nu,\nu-j-1}(z,t)}{j!(z-t)^{\nu-j}} = \sum_{j=0}^{\nu-1} F^{(j)}(t)\frac{G_{\nu,j}(z,t)}{B(z)^\nu}.
}

Now, for a given complex number $w\in\C$ we define another linear operator $\M_w$ on the space of entire functions $F$ by
\est{
\M_w(F)(z)=\frac{F(z)B(w)^\nu-B(z)^\nu F(w)}{z-w}.
}
We observe that for every $t\in\Tau_B$ and every $w\in\C\setminus \Tau_B$ we have
\es{\label{id-Mw-Sw}
\Sg_t\bigg(\frac{\M_w(F)(\cdot)}{B(w)^\nu B(\cdot)^\nu}\bigg)(z) = \frac{\Sg_t(F/B^\nu)(z)-\Sg_t(F/B^\nu)(w)}{z-w}.
 }
One can deduce this last identity by observing that  
\est{
\frac{\M_w(F)(z)}{B(w)^\nu B(z)^\nu}-\frac{\Sg_t(F/B^\nu)(z)-\Sg_t(F/B^\nu)(w)}{z-w}
}
has a removable singularity at the point $z=w$ and also that the right hand side of \eqref{id-Mw-Sw} is a rational function in the variable $z$ with exactly one pole located at $z=t$ and it satisfies condition \eqref{eq-4}.

{ \it Step 1.} We begin with the case $p\in[1,2)$. The assumption that $E'/E\in H^\infty(\C^+)$ implies with Theorem \ref{Thm1} that $\H^2(E^\nu)$ is closed under differentiation. Since $\p'(x)=\re i\frac{E'(x)}{E(x)}$, we can apply Lemma \ref{inclusion-lemma} to conclude that $\H^p(E^\nu)\subset \H^2(E^\nu)$, hence formula \eqref{Lp-recont-form} is a direct consequence of Theorem \ref{Thm2}.

{\it Step 2.}  Now, we deal with the case $p\in(2,\infty)$. A crucial observation is that if $F\in\H^p(E^\nu)$ then $\M_w(F)\in\H^2(E^\nu)$. Thus, we can apply Theorem \ref{Thm2} together with \eqref{eq-5} to obtain
$$
\M_w(F)(z) = \sum_{t\in\Tau_B} B(z)^\nu\Sg_t\bigg(\frac{\M_w(F)(\cdot)}{B(\cdot)^\nu}\bigg)(z)
$$
where the last sum converges uniformly in the variable $z$ in every compact subset of $\C$ for every fixed $w\in\C$. By \eqref{id-Mw-Sw}, we conclude that
\es{\label{preliminar-int-form}
\frac{F(z)}{B(z)^\nu}-\frac{F(w)}{B(w)^\nu} =  \sum_{t\in\Tau_B} \bigg\{ \sum_{j=0}^{\nu-1} F^{(j)}(t)\frac{G_{\nu,j}(z,t)}{B(z)^\nu}-\sum_{j=0}^{\nu-1} F^{(j)}(t)\frac{G_{\nu,j}(w,t)}{B(w)^\nu}\bigg\}
}
for every $w,z\in\C\setminus\Tau_B$. We claim that \eqref{preliminar-int-form} implies that
\es{\label{partial-rec-form}
F(z) = \Lambda(F)B(z)^\nu + \sum_{t\in\Tau_B}  \sum_{j=0}^{\nu-1} F^{(j)}(t)G_{\nu,j}(z,t)
}
for some constant $\Lambda(F)$, where the sum converges uniformly in compact sets of $\C$. Assuming that \eqref{partial-rec-form} is valid, clearly the map $F\mapsto \Lambda(F)$ defines a linear functional in the space $\H^p(E^\nu)$. In the next steps we will show that formula \eqref{partial-rec-form} holds and that $F\mapsto \Lambda(F)$ is a continuous functional that vanishes in a dense set of functions in $\H^p(E^\nu)$ hence it vanishes identically. This would conclude the proof.

{\it Step 3.} Recall that $a_{\nu,j}(t)$ is defined as the coefficients of the Taylor expansion of $B_{\nu,\nu}(z,t)^{-1}$ at the point $z=t$ with $t\in\Tau_B$. By item (3) Lemma \ref{est-inner-prod-norm-B} these coefficients satisfy the estimate
$$
|a_{\nu,j}(t)|^2 \lesssim \frac{1}{K_\nu(t,t)}
$$
for $j=0,...,\nu-1$, where $\D$ is the norm of the differentiation operator in $\H^2(E^\nu)$ and $K_\nu(w,z)$ is the reproducing kernel associated with $E(z)^\nu$. Since
$$
G_{\nu,j}(z,t)=\frac{B(z)^\nu}{j!}\sum_{\ell=0}^{\nu-j-1}\frac{a_{\nu,\ell}(t)}{(z-t)^{\nu-\ell-j}},
$$
we obtain
\es{\label{eq-7}
K_\nu(t,t)|G_{\nu,j}(i,t)|^2 \lesssim \frac{1}{1+t^2}
}
for every $t\in\Tau_B$ and $j=0,...,\nu-1$.

The condition $E'/E\in H^\infty(\C^+)$ implies, by Proposition \ref{equiv-con-Lp-prop}, that $\H^p(E^\nu)$ is closed under differentiation and that the zeros of $E$ are separated from the real line. Also, the assumption that $\p'(t)\geq \delta$ for all $t\in\Tau_B$ together with formula \eqref{phi-id} implies that
\es{\label{eq-9}
|E(t)|^{2\nu}\simeq K_\nu(t,t),
}
for all $t\in\mc{T}_B$.  From the condition $E'/E\in H^\infty(\C^+)$ we obtain with Theorem \ref{Thm1} and \eqref{zero-separation} that $\Tau_B$ is a sequence of uniformly separated points. We can now apply Proposition \ref{PP-lemma} together with \eqref{eq-9} to obtain
\es{\label{eq-10}
\sum_{t\in\Tau_B} \bigg|\frac{F(t)}{K_\nu(t,t)^{1/2}}\bigg|^{p} \lesssim \int_{\R}\bigg|\frac{F(x)}{E(x)^\nu}\bigg|^p\dx
}
for every $F\in\H^p(E^\nu)$. Finally, we obtain the following estimate
\es{\label{conv-est}
\sum_{j=0}^{\nu-1} \sum_{t\in\Tau_B}|F^{(j)}(t)G_{\nu,j}(i,t)|
& \leq \sum_{j=0}^{\nu-1}  \bigg[\sum_{t\in\Tau_B}\bigg|\frac{F^{(j)}(t)}{K_\nu(t,t)^{1/2}}\bigg|^{p} \bigg]^{1/p}\bigg[\sum_{t\in\Tau_B} \bigg|\frac{G_{\nu,j}(i,t)}{K_\nu(t,t)^{-1/2}}\bigg|^{p'} \bigg]^{1/p'} \\
& \lesssim \sum_{j=0}^{\nu-1}  \bigg[\sum_{t\in\Tau_B}\bigg|\frac{F^{(j)}(t)}{K_\nu(t,t)^{1/2}}\bigg|^{p}\bigg]^{1/p} \\ 
& \lesssim \|F(x)/E(x)^\nu\|_{L^p},
}
where the first inequality is H\"older's inequality, the second one due to \eqref{eq-7} and the separation of $\Tau_B$, the third one due to \eqref{eq-10} and the closure under differentiation of $\H^p(E^\nu)$. 

{\it Step 4.} Estimate \eqref{conv-est} together with formula \eqref{preliminar-int-form} for $w=i$ clearly implies that \eqref{partial-rec-form} is valid. We can use H\"older's inequality \eqref{holder-Lp} together with \eqref{conv-est} again to conclude that $F\mapsto \Lambda(F)$ is a continuous functional over $\H^p(E^\nu)$. By Lemmas \ref{BGdense} and \ref{inclusion-lemma} the functions $\{G_{\nu,j}(z,t)\}$ for $j=0,...,\nu-1$ and $t\in\Tau_B$ form a dense set in $\H^p(E^\nu)$ and trivially $\Lambda(G_{\nu,j}(z,t))=0$. Hence $\Lambda$ vanishes identically. This concludes the proof.


\end{document}